\documentclass[10pt,a4paper]{amsart}

\usepackage[all,cmtip]{xy}
\usepackage{amssymb}

\theoremstyle{plain}
\newtheorem{theorem}{Theorem}[section]
\newtheorem{lemma}[theorem]{Lemma}
\newtheorem{proposition}[theorem]{Proposition}
\newtheorem{corollary}[theorem]{Corollary}
\theoremstyle{definition}

\theoremstyle{remark}
\newtheorem{remark}[theorem]{Remark}

\newcommand{\ac}{\ensuremath{\mathcal{A}}}
\newcommand{\fc}{\ensuremath{\mathcal{F}}}
\newcommand{\dc}{\ensuremath{\mathcal{D}}}

\newcommand{\hc}{\ensuremath{\mathcal{H}}}

\newcommand{\tc}{\ensuremath{\mathcal{T}}}

\newcommand{\bc}{\ensuremath{\mathcal{B}}}

\newcommand{\As}{\mathbb{A}}

\def\bin #1#2 {\left( \matrix { #1 \cr #2 \cr } \right) }

\begin{document}

\title[Bivariant class of degree one]
{Bivariant class of degree one}

\author{Vincenzo Di Gennaro }
\address{Universit\`a di Roma \lq\lq Tor Vergata\rq\rq, Dipartimento di Matematica,
Via della Ricerca Scientifica, 00133 Roma, Italy.}
\email{digennar@axp.mat.uniroma2.it}

\author{Davide Franco }
\address{Universit\`a di Napoli
\lq\lq Federico II\rq\rq, Dipartimento di Matematica e Applicazioni
\lq\lq R. Caccioppoli\rq\rq, Via Cintia, 80126 Napoli, Italy.}
\email{davide.franco@unina.it}

\author{Carmine Sessa }
\address{Universit\`a di Napoli
\lq\lq Federico II\rq\rq, Dipartimento di Matematica e Applicazioni
\lq\lq R. Caccioppoli\rq\rq, Via Cintia, 80126 Napoli, Italy.}
\email{carmine.sessa2@unina.it}

\abstract Let $f:X\to Y$ be a projective birational morphism,
between complex quasi-projective varieties. Fix a bivariant class
$\theta \in H^0(X\stackrel{f}\to Y)\cong
Hom_{D^{b}_{c}(Y)}(Rf_*\mathbb A_X, \mathbb A_Y)$ (here $\mathbb A$
is a Noetherian commutative ring with identity, and $\mathbb A_X$
and $\mathbb A_Y$ denote the constant sheaves). Let
$\theta_0:H^0(X)\to H^0(Y)$ be the induced Gysin morphism. We say
that {\it $\theta$ has degree one} if $\theta_0(1_X)= 1_Y\in
H^0(Y)$. This is equivalent to say that $\theta$ is a section of the
pull-back $f^*: \mathbb A_Y\to Rf_*\mathbb A_X$, i.e. $\theta\circ
f^*={\text{id}}_{\mathbb A_Y}$, and it is also equivalent to say
that $\mathbb A_Y$ is a direct summand of $Rf_*\mathbb A_X$. We
investigate the consequences of the existence of a bivariant class
of degree one. We prove explicit formulas relating the (co)homology
of $X$ and $Y$, which extend the classic formulas of the blowing-up.
These formulas are compatible with the duality morphism. Using
which, we prove that the existence of a bivariant class $\theta$ of
degree one for a resolution of singularities, is equivalent to
require that $Y$ is an $\mathbb A$-homology manifold. In this case
$\theta$ is unique, and  the Betti numbers of the singular locus
${\text{Sing}}(Y)$ of $Y$ are related with the ones of
$f^{-1}({\text{Sing}}(Y))$.



\endabstract
\maketitle

\bigskip

\section{Introduction}

Let $\mathbb A$ be a Noetherian commutative ring with identity. All
the (co)homology groups  occurring in this paper will be with
$\mathbb A$-coefficients.

Consider a resolution of singularities $f:X\to Y$ of a complex
quasi-projective variety $Y$ of dimension $n$. When $\mathbb A$ is a field,
the Decomposition Theorem \cite[p. 161]{Dimca2} implies there exists a certain
decomposition
\begin{equation}\label{splittingintro}
Rf_*\mathbb A_X[n] \cong  IC_Y^{\bullet} \oplus \mathcal H
\end{equation}
in $D_{c}^{b}(Y)$, the derived category of bounded constructible
complexes of $\mathbb A$-sheaves in $Y$. If $Y$ is an $\mathbb
A$-homology manifold \cite{Brasselet}, then $IC_Y^{\bullet}\cong \mathbb A_Y[n]$.
Hence, we get
\begin{equation}\label{splittingintro2}
Rf_*\mathbb A_X[n] \cong \mathbb A_Y[n] \oplus \mathcal H.
\end{equation}

One of our purposes is to extend the splitting
(\ref{splittingintro2}) to every ring $\mathbb A$, for which the
Decomposition Theorem providing (\ref{splittingintro}) is not
available. Specifically, we will see that the splitting
(\ref{splittingintro2}) is equivalent to the existence of a
\textit{bivariant class of degree one} for $f$, that we are about to
define.

Bivariant Theory was introduced in early 1980 by W. Fulton and R.
MacPherson \cite{FultonCF}, for the purpose of unifying covariant
and contravariant theories. The sheaf-theoretic bivariant homology
theory associates to a continuous map of topological spaces
$X\stackrel{f}{\to} Y$,  the graded group with homogeneous
components
\begin{equation*}
H^i (X\stackrel{f}{\to}  Y)={\rm{Hom}}_{D^b_c(Y)}(f_!\mathbb A_X,
\mathbb A_Y[i]),
\end{equation*}
whose elements are called bivariant classes. Bivariant Theory allows
a systematic study of generalized wrong-way Gysin morphisms. These
morphisms find a great use especially in the study of morphisms of
smooth varieties or, more generally, of locally complete
intersection morphisms.

In some cases, a bivariant class determines a very interesting
splitting in the derived category \cite[p. 327]{Jouanolou}. Another
purpose of this paper is to show that a similar splitting can be
proved in a more general context, and that the natural definition
involved is that of bivariant class of degree one. It is worthy to
stress that such a splitting turns out to be compatible with
Poincar\'e  Duality (Corollary 5.1). Consider a bivariant class
belonging to the $0$-th homogeneous component (now we assume $f:X\to
Y$ is proper):
$$\theta \in H^0(X\stackrel{f}\to Y)={\rm{Hom}}_{D^b_c(Y)}(f_*\mathbb A_X, \mathbb A_Y).$$ Let
$\theta_0:H^0(X)\to H^0(Y)$  be the induced Gysin morphism. We say
that $\theta$ \textit{has degree one} (for $f$) if
$\theta_0(1_X)=1_Y$. This is equivalent to say that $\theta$ is a
section of the pull-back $f^*: \mathbb A_Y\to Rf_*\mathbb A_X$, i.e.
$\theta\circ f^*={\text{id}}_{\mathbb A_Y}$. We will see that the
existence of a bivariant class of degree one leads to a suitable
splitting in the derived category. Consequently, we deduce a series
of isomorphisms for (co)homology groups, which extend classic
formulas of the blowing-up, and that we have extensively used in
Noether-Lefschetz Theory \cite{RCMP}.

Examples of morphisms admitting a bivariant class of degree one are
blowing-ups at locally complete intersection subvarieties. Indeed,
the {\it orientation class} of the blowing-up \cite[p.
114]{FultonIT}, \cite[p. 131]{FultonCF} is a bivariant class of
degree one (Remark \ref{remintro}, $(iii)$). Other examples are
\textit{strong orientation classes} (of codimension $0$) $\theta\in
H^0(X\stackrel{f}\to Y)$ \cite[p. 27]{FultonCF}, \cite[p.
803]{Brasselet}, for maps $f$ between varieties of the same
dimension (Corollary 6.3). We will see that the class of birational
morphisms admitting a bivariant class of degree one is considerably
broader than the class admitting strong orientation classes. For
instance, any blowing-up at a locally complete intersection
subvariety admits a bivariant class $\theta$ of degree one, but it
is rather unlikely that $\theta$ is a strong orientation when the
center is not smooth (Remark 6.5).

We study this circle of questions for a morphism $f: X\to Y$
which is a resolution of singularities of a complex quasi-projective
variety $Y$ or, more generally, for a morphism from an $\mathbb
A$-homology manifold $X$, which is an isomorphism on a non-empty
open subset $U\cong f^{-1}(U) \subset X$. Our main results are
collected in the following two theorems. The first one, together
with its consequences for the (co)homology (see Section $4$), should
be compared  with \cite[p. 114-118]{FultonIT}, \cite[p.
327]{Jouanolou}, \cite[p. 263]{RCMP}, where similar results appear
in the study of the behavior of the (co)homology and of the Chow
groups, under blowing-up at a locally complete intersection
subvariety of a quasi-projective variety. The second theorem gives,
as far as we know, a new characterization of homology manifolds, in
terms of their resolution of singularities.

\medskip
\begin{theorem}
\label{Crossintro} Let $f:X\to Y$ be a  continuous and proper map,
with $Y$ path-connected. Let $U\subseteq Y$ be a non-empty open
subset such that $f$ induces an homeomorphism $f^{-1}(U)\cong U$.
Set $W= Y\backslash U$, and $\widetilde W=f^{-1}(W)$. The following
properties are equivalent.

\smallskip
$(i)$ There exists a bivariant class $\theta\in
Hom_{D^{b}_{c}(Y)}(Rf_*\mathbb A_X, \mathbb A_Y)$ of degree one.

\smallskip $(ii)$ In $D^{b}_{c}(Y)$ there exists a cross isomorphism
$Rf_*\mathbb A_X\oplus \mathbb A_W\cong Rf_*\mathbb
A_{\widetilde W}\oplus \mathbb A_Y$.

\smallskip
$(iii)$ In $D^{b}_{c}(Y)$ there exists a decomposition
$Rf_*\mathbb A_X \cong \mathbb A_Y \oplus \mathcal K$.
\end{theorem}

\medskip
\begin{theorem}\label{homology}

Let $f:X\to Y$ be a projective birational morphism between
complex, irreducible, and quasi-projective varieties
of the same dimension $n$. Let $U$ be a non-empty Zariski open subset of $Y$
such that $f$ induces an isomorphism $f^{-1}(U)\cong U$. Set $W=Y\backslash U$.

$\bullet$ If  $Y$ is an  $\mathbb
A$-homology manifold, then there exists a bivariant class $\theta$
in $Hom(Rf _* \mathbb A _X , \mathbb A_Y)$
of degree one. In this
case, $\theta$ is unique, and there exists a decomposition
$Rf _* \mathbb A _X\cong \mathbb A _Y\oplus \mathcal K$,
with $\mathcal K$ supported on $W$. Moreover, if also $X$ is an
$\mathbb A$-homology manifold, then $\mathcal K[n]$ is self-dual.

$\bullet$
Conversely, if $X$ is an $\mathbb A$-homology manifold and there
exists a bivariant class $\theta \in Hom(Rf _* \mathbb A _X ,
\mathbb A_Y)$ of degree one, then  also $Y$ is an $\mathbb
A$-homology manifold.
\end{theorem}

Theorem \ref{Crossintro} follows from more general results that hold
true in any triangulated categories (Lemma 3.2, Lemma 3.5). The
decompositions $(ii)$ and $(iii)$ in Theorem \ref{Crossintro} induce
explicit isomorphisms in (co)homology (Section 4), that are
compatible with the cap-product with the fundamental class (Section
5). Using which, one may easily prove Theorem \ref{homology}. Since
$\mathcal K[n]$ is self-dual, it follows that the Betti numbers of
the singular locus ${\text{Sing}}(Y)$ of $Y$, and of
$f^{-1}({\text{Sing}}(Y))$, are related (Remark \ref{remfin},
$(ii)$).

Other results are obtained along the way. Two of them seem  to us
worthy to note.
\begin{enumerate}
\item Suppose that the birational morphism $f: X\to Y$ admits a strong orientation class  $\theta\in
H^0(X\stackrel{f}\to Y)$. If one between X and Y is an
$\As$-homology manifold, then the other is too (Theorem
\ref{homology}, Corollary 6.3, Proposition 6.4). In this case, every
birational morphism between $X$ and $Y$ admits a strong orientation
class.
\item There are examples of  projective
birational maps $f:X\to Y$ such that $H^0(X\stackrel{f}\to Y)\neq
0$, without bivariant classes of degree one (Remark 6.2, $(iii)$).
\end{enumerate}

\bigskip
\section{Notations.}

\bigskip
$(i)$ Let $\mathbb A$ be a Noetherian commutative ring with identity
(e.g. $\mathbb A=\mathbb Z$ or $\mathbb A=\mathbb Q$). Every
topological space $V$ occurring in this paper will be assumed to be
imbeddable as a closed subspace of some $\mathbb R^N$ \cite[p.
32]{FultonCF} (e.g. a complex quasi-projective
variety, with the natural topology, and its open subsets).
Maps between topological spaces are assumed continuous of finite
cohomological dimension \cite[p. 83]{FultonCF} (e.g.
algebraic maps between  complex quasi-projective
varieties, and their restrictions on open subsets). We denote by
$H^{i}(V)$ and $H_{i}(V)$ the cohomology and the Borel-Moore
homology groups, with $\mathbb A$-coefficients, of $V$
\cite{FultonYT}. We denote by ${\text{Sh}}(V)$ the category of
sheaves of $\mathbb A$-modules on $V$. Let $D_{c}^{b}(V)$ denote the
derived category of bounded constructible complexes of $\mathbb
A$-sheaves $\mathcal{F^{\bullet}}$ on $V$ \cite{Dimca2},
\cite{DeCM3}. The symbol $IC_{V}^{\bullet}$ represents the
intersection cohomology complex of $V$. If $V$ is a smooth,
irreducible, quasi-projective complex variety of dimension $n$, then
$IC_{V}^{\bullet}\cong \mathbb A_V[n]$, where $\mathbb A_V$ is the
constant sheaf.

\medskip $(ii)$ Let $f:X\to Y$ be a continuous and proper map.
Fix a bivariant class \cite{FultonCF}
$$\theta \in H^0(X\stackrel{f}\to Y)\cong
Hom_{D^{b}_{c}(Y)}(Rf_*\mathbb A_X, \mathbb A_Y).$$ Let
$\theta_0:H^0(X)\to H^0(Y)$ be the induced Gysin homomorphism. We
say that {\it $\theta$ has degree one} (for the map $f$) if
$\theta_0(1_X)= 1_Y\in H^0(Y)$ \cite[p. 238]{Spanier}.

\medskip $(iii)$ Let $V$ be an irreducible, quasi-projective variety
of complex dimension $n$. We say that $V$ is an {\it $\mathbb A$-homology
manifold} if for all $y\in Y$ and for all $i\neq 2n$ one has
$H_{i}(Y, Y\backslash\{y\})=0$, and $H_{2n}(Y,
Y\backslash\{y\})\cong \mathbb A$ \cite{BM}, \cite{Brasselet} (by
$H_{i}(Y, Y\backslash\{y\})$ we denote the singular homology of a
pair). This is equivalent to say that $\mathbb A_Y[n]$ is self-dual,
or that $\mathbb A_Y[n]\cong IC_{Y}^{\bullet}$ \cite[p. 804-805]{Brasselet}.

\medskip $(iv)$ An element $\theta\in H^i(X\stackrel{f}\to Y)$
is called a {\it strong orientation of codimension $i$} for the morphism
$f:X\to Y$ if, for all morphisms $g:Z\to X$, the morphism
$$
 H^{\bullet}(Z\stackrel{g}\to X)\stackrel{\bullet\,\theta}\to H^{\bullet}(Z\stackrel{f\circ g}\to Y)
$$
is an isomorphism \cite[p. 26]{FultonCF}, \cite[p. 803]{Brasselet}.

\bigskip
\begin{remark}\label{remintro}
$(i)$ Observe that  {\it $\theta$ has degree one if and only if
$\theta$ is a section of the pull-back $f^*:\mathbb A_Y\to
Rf_*\mathbb A_X$}, i.e.
$$
\theta_0(1_X) =1_{Y} \iff \theta\circ f^*={\text{id}}_{\mathbb A_Y}.
$$

In fact,  assume that $\theta$ is of degree one.   For every $y\in
H^{\bullet}(Y)$ one has (\cite[p. 26, (G4), (i)]{FultonCF}, \cite[Spanier,
p. 251, 9]{Spanier}):
$$
\theta_*(f^*(y))=\theta_*({1_X}\cup f^*(y))=\theta_*(1_X) \cup
y=1_{Y}\cup y=y
$$
for every $y\in H^{\bullet}(Y)$. By functoriality, this means that the
morphism $\theta \circ f^*$ induces the identity on the
cohomology groups ${{\rm{id}}_{H^{\bullet}(Y)}}=\theta_*\circ f^*:H^{\bullet}(Y)\to
H^{\bullet}(Y)$. On the other hand, we have $ \theta\circ f^*\in
{\rm{Hom}}_{D^b_c(Y)}(\mathbb A_Y, \mathbb A_Y)\cong H^0(Y).$ It
follows that $\theta\circ f^*={\text{id}}_{\mathbb A_Y}$.

Conversely, if $\theta\circ f^*={\text{id}}_{\mathbb A_Y}$, then the
composite $ H^0(Y)\stackrel{f^*}\to H^0(X)\stackrel{\theta_0}\to
H^0(Y)$ is the identity of $H^0(Y)$. Since $f^*(1_Y)=1_X$, it
follows that $\theta_0(1_X)=1_{Y}$, i.e. $\theta$ has degree one.

\smallskip
$(ii)$ Let $f:X\to Y$ be a proper map. Let $\theta
\in H^0(X\stackrel{f}\to Y)$ be a bivariant class. If
$\theta_0(1_X)=d\cdot 1_Y\in H^0(Y)$, and if $d$ is a unit in
$\mathbb A$, then $d^{-1}\cdot\theta$ is a bivariant class of degree
one. Moreover, let $i:W\subseteq Y$ be a non-empty subspace of $Y$,
and let $g:f^{-1}(W)\to W$ be the restriction of $f$ on $f^{-1}(W)$.
Denote by $\theta'=i^*(\theta)\in H^0(f^{-1}(W)\stackrel{g}\to W)$
the pull-back of $\theta$. By \cite[(G2), $(ii)$, p. 26]{FultonCF}, we
see that $i^*\theta_0(1_X)=\theta'_0j^*(1_X)$, where
$j:f^{-1}(W)\subseteq X$ denotes the inclusion. Therefore,
$1_W=\theta'_0(1_{f^{-1}(W)})\in H^0(W)$. This proves that the
pull-back of a bivariant class of degree one, is again of degree
one. And, conversely, if $Y$ is path-connected, and  $\theta'$ is of
degree one, then also $\theta$ is of degree one.

\smallskip
$(iii)$ Assume that $f:X\to Y$ is a projective, locally complete
intersection morphism between complex irreducible quasi-projective
varieties, and that $f$ is birational (e.g. $f$ is the blowing-up of
$Y$ at a locally complete intersection subvariety $W\subset Y$
\cite[p. 114]{FultonIT}). Let $\theta\in H^0(X\stackrel{f}\to Y)$ be
the {\it orientation class} of $f$ \cite[p. 114]{FultonIT}, \cite[p.
131]{FultonCF}. Then $\theta$ has degree one. In fact, let $U$ be a
non-empty Zariski open set of $Y$, such that $f$ induces an
isomorphism $f^{-1}(U)\cong U$. Let $\theta'$ be the restriction of
$\theta$ on $f^{-1}(U)\to U$. Since $\theta'$ is the orientation
class of $f^{-1}(U)\to U$ \cite[Lemma 19.2, (a), p. 379]{FultonIT},
and $f^{-1}(U)\cong U$, it follows that $\theta'$ has degree one. By
remark $(ii)$ above,  also $\theta$ has degree one . Compare with
\cite[p. 137]{BFMP} and \cite[p. 12]{Verdier}.

\smallskip
$(iv)$ If $Y$ is a quasi-projective {$\mathbb A$-homology manifold},
and $f: X\rightarrow Y$ is a {resolution of singularities}  of $Y$,
then there exists a {unique} bivariant class $\theta\in
Hom_{D^{b}_{c}(Y)}(Rf_*\mathbb A_X, \mathbb A_Y)$ of degree one. See
Theorem \ref{homology} above.

\smallskip
$(v)$ Let $f:X\to Y$ be a projective map between irreducible,
complex quasi-projective varieties of the same dimension $n$. Assume
that $Y$ is smooth (or, more generally, that $Y$ is an $\mathbb
A$-homology manifold). In this case one has (compare with
\cite[3.1.4, p. 34]{FultonCF}, \cite[Lemma 2, p. 217]{FultonYT}, and
the proof of Theorem \ref{homology} below):
$$
H^0(X\stackrel{f}\to Y)\cong H_{2n}(X)\cong H^0(X).
$$
By remark $(i)$ above, if there exists a bivariant class of degree
one for $f$, then, for every $k$,  $H^k(Y)$ is contained, via
pull-back, in $H^k(X)$. Therefore, if $\mathbb A=\mathbb Z$ and
$h^k(Y)>h^k(X)$ for some $k$, then it  happens that
$H^0(X\stackrel{f}\to Y)\neq 0$, but $\theta_0=0$, for every
bivariant class $\theta$.  However, if, in addition, $f$ is
birational, then the bivariant class $\theta$ corresponding to
$1_X\in H^0(X)$ is a bivariant class of degree one. In fact, if $U$
is a Zariski open subset of $Y$ such that $f^{-1}(U)\cong U$, the
restriction of $\theta$ on $f^{-1}(U)\to U$ has degree one. Observe
that, if $Y$ is singular, it is no longer true. For instance, let
$C\subset \mathbb P^3$ be a projective non-singular curve of genus
$\geq 1$. Let $Y\subset \mathbb P^4$ be the cone over $C$, and let
$f:X\to Y$ be the blowing-up of $Y$ at the vertex. Then one has
$H^0(X\stackrel{f}\to Y)\neq 0$, but there is no a bivariant class
of degree one of $f$. This is a consequence of Theorem
\ref{homology}. For more details, see Remark \ref{remfin}, $(iii)$.

\smallskip
$(vi)$ Let $f:X\to Y$ be a projective map between irreducible
quasi-projective varieties. Assume there exists a bivariant class
$\theta$ of degree one. Put $n=\dim X$, and $m=\dim Y$. Since
$f_*\circ \theta^*={\rm{id}_{H_{\bullet}(Y)}}$, the push-forward map
$f_*$ induces an inclusion $H_{\bullet}(Y)\subseteq H_{\bullet}(X)$.
It follows that $m\leq n$. Moreover, $f$ is surjective, otherwise
the push-forward $f_*:H_{2m}(X)\to H_{2m}(Y)$ vanishes. Since
restricting $\theta$ to some special fibre, we obtain again a
bivariant class of degree one, in general it may happen that $n>m$.
It is clear that, if $n=m$, then $f$ is birational.
\end{remark}

\bigskip
\section{Bivariant class of degree one and decompositions.}

\bigskip
In this section we are going to prove Theorem \ref{Crossintro} stated
in the Introduction.

To this purpose, we need some preliminaries.
The first one is the following lemma.

\begin{lemma}\label{l1} Let $\mathcal T$ be a triangulated category, and  $f^*\in
{\rm{Hom}}_{\mathcal T}(A, B)$ be a morphism in $\mathcal T$. Assume
that $f^*$ if left-invertible, i.e. that there exists $\theta \in
{\rm{Hom}}_{\mathcal T}(B, A)$ such that $\theta\circ f^*=1_A$. Then
we have $B\cong A\oplus C$ for some $C\in Ob (\mathcal T)$.
\end{lemma}
\begin{proof}[Proof of Lemma \ref{l1}]
The axiom TR1 $(iii)$ of triangulated categories implies that $f^*$
can be completed to a distinguished triangle
$$A \stackrel{f^*}{\longrightarrow} B \longrightarrow C $$
\cite[p. 12]{Huy}. Thus, combining the hypothesis $\theta\circ
f^*=1_A$ with axioms TR1 and TR3, we have a commutative diagram of
distinguished triangles
\[
\xymatrix{
A \ar[r]^{f^*} \ar[d]^{1_{A}}& B \ar[d]^{\theta} \ar[r]& C  \ar[d]^{}\\
A  \ar[r]^{1_A} & A \ar[r]  & 0. }
\]
The axiom TR2 provides also the following commutative diagram of
distinguished triangles
\[
\xymatrix{
C \ar[r]^{\delta} \ar[d]& A[1] \ar[d]^{1_{A[1]}} \ar[r]& B[1]  \ar[d]^{\theta[1]}\\
0  \ar[r] & A[1] \ar[r]  & A[1], }
\]
from which we argue that $\delta$ vanishes.  We conclude at once by
\cite[Exercise 1.38]{Huy}.
\end{proof}

\smallskip
We are in position to prove that $(i)$ is equivalent to $(iii)$ in
Theorem \ref{Crossintro}.

To this purpose, first assume there exists a bivariant class
$\theta: Rf_*\mathbb A_X\to \mathbb A_Y$ of degree one, and let
$f^*:\mathbb A_Y\to Rf_*\mathbb A_X$ be the pull-back morphism. By
Remark \ref{remintro}, $(i)$, we know that $\theta\circ
f^*=1_{\mathbb A_Y}$. Therefore, we may apply previous Lemma
\ref{l1}, with $\mathcal T=D^b_c(Y)$, $A=\mathbb A_Y$,
$B=Rf_*\mathbb A_X$, with the morphism $f^*$ as the pull-back, and
$\theta$ as the given bivariant class. It follows a decomposition
like $Rf_*\mathbb A_X \cong \mathbb A_Y \oplus \mathcal K$.

Conversely, suppose there exists a decomposition $Rf_*\mathbb A_X
\cong \mathbb A_Y \oplus \mathcal K$. By projection, it induces a
bivariant class $\eta: Rf_*\mathbb A_X \to \mathbb A_Y$. Since the
restriction $\eta'$ of $\eta$ on $U$ is an automorphism of ${\mathbb
A}_U$, and $U$ is nonempty, it follows that $\eta'_0(1_U)=d\cdot
1_U\in H^0(U)$,  with some unit $d\in\mathbb A$. Therefore,
$d^{-1}\cdot \eta$ is a bivariant class of degree one (compare with
Remark \ref{remintro}, $(ii)$).

This concludes the proof that $(i)$ is equivalent to $(iii)$ in
Theorem \ref{Crossintro}.

\medskip
\begin{remark}\label{nonempty}
In order to prove that $(i)$ implies $(iii)$, we do not need the
existence of $U$.
\end{remark}

Now we are going to prove that $(i)$ is equivalent to $(ii)$.

\smallskip
Observe that the same argument we just used to prove that $(iii)$
implies $(i)$,  proves that $(ii)$ implies $(i)$. In fact, suppose
there exists a decomposition $Rf_*\mathbb A_X\oplus \mathbb A_W\cong
Rf_*\mathbb A_{\widetilde W}\oplus \mathbb A_Y$. By projection, it
induces a bivariant class $\eta: Rf_*\mathbb A_X \to \mathbb A_Y$.
Since both $\mathbb A_W$ and $Rf_*\mathbb A_{\widetilde W}$ are
supported on $W$, the restriction of $\eta$ on $U$ is an
automorphism of ${\mathbb A}_U$. And now we may conclude as before.

\smallskip
In order to conclude the proof of Theorem \ref{Crossintro}, we only have
to prove that $(i)$ implies $(ii)$. Also in this case, we need some
preliminaries.

Consider the following natural commutative diagram
\begin{equation}\label{diagram}
\xymatrix{
\widetilde{W} \ar[d]^g \ar[r]^j& X \ar[d]^f &U \ar[l]_{\partial_X} \ar[d]^1\\
W \ar[r]^i &Y  & U \ar[l]_{\partial_Y} }
\end{equation}
where $g:\widetilde W\to W$ denotes the restriction of $f$,  and the
other maps are the inclusions. Denote by $\ac$ (resp. $\bc$)  the
full subcategory of ${\text{Sh}}(X)$ (resp. ${\text{Sh}}(Y)$)
supported on $U$.
\begin{lemma}\label{l2}
\label{equivalence} On the category ${\rm{Sh}}(U)$ we have $f_*
\circ
\partial_{X\, !}=\partial_{Y\, !}$. Furthermore, $f_*$ is an exact
equivalence between $\ac$ and $\bc$, whose inverse is the pull-back
$f^*$.
\end{lemma}
\begin{proof}
First we prove that $f_* \circ    \partial_{X\, !}=\partial_{Y\, !}$
on ${\text{Sh}}(U)$.

Let $\fc$ be a sheaf on $U$   and let $V\subseteq Y$ be an open
subset. By \cite[Definition 6.1, p. 106]{Iversen}, we have
$$f_*(\partial_{X !}(\fc))(V)= \{s\in \Gamma\left( f^{-1}(V) \cap U, \fc\right) \mid  \,\, \text{supp($s$) is closed in} \,\, f^{-1}(V)\}$$
and
$$\partial_{Y !}(\fc)(V)= \{s\in \Gamma\left( V\cap U, \fc\right) \mid  \,\, \text{supp($s$) is closed in} \,\, V\}.$$
Since $f$ is continuous, we have $\partial_{Y !}(\fc)(V)\subseteq
f_*(\partial_{X !}(\fc))(V)$. Hence, $ \partial_{Y !}(\fc)$ is a
subsheaf of  $f_*(\partial_{X !}(\fc))$.
As for the opposite inclusion, we argue as follows. By the local
compactness of $Y$, we can assume that the closure of $V$ is compact
in $Y$. Fix $s\in f_*(\partial_{X !}(\fc))(V)$, and set $C:=
{\rm{supp}}(s)$, so that $C$ is closed in $f^{-1}(V)$. It suffices
to prove that $f(C)$, which is homeomorphic to $C$, is closed in
$V$. Since  $f$ is a proper morphism, $f^{-1}(\overline V )$ is
compact and  the map $f^{-1}(\overline V ) \rightarrow \overline V$
is closed. Then we have
$$C=f(C)= f(\overline C \cap f^{-1}(V))=f(\overline C) \cap V $$
and we are done.

We are left with the proof that $f_*$ induces an exact equivalence
between $\ac$ and $\bc$. By \cite[Proposition 6.4, p. 107]{Iversen},
we already know that $f_*$ induces an equivalence between $\ac$ and
$\bc $, whose inverse is the pull-back.
As for the exactness, $f_*$ first of all is left-exact by \cite[p.
97]{Iversen}. Now, consider an exact sequence of sheaves in $\ac$:
$\dc \rightarrow \hc \rightarrow 0$. By \cite[Proposition 6.4, p.
107]{Iversen}, we can assume
$\dc= \partial_{X !}\dc _U, \,\,\,\, \hc= \partial_{X !}\hc _U$,
for suitable and well determined sheaves $\dc _U, \, \hc _U \in
{\rm{Sh}}\, (U)$. Therefore, taking into account we just proved that
$f_* \circ
\partial_{X\, !}=\partial_{Y\, !}$,
by \cite[(6.3) p. 106]{Iversen} we deduce
\[
\xymatrix{
f_*\partial_{X !}\dc _U \ar[d] _= \ar[r]& f_*\partial_{X !}\hc _U  \ar[d]^= \ar[r] & 0  \ar[d]\\
\partial_{Y !}\dc _U \ar[r] & \partial_{Y !}\hc _U \ar[r] & 0
}
\]
and we are done.
\end{proof}

\begin{lemma}\label{l3}
Consider a triangulated category $\tc$,  and two  commutative
diagram of distinguished triangles in $\tc$
$$
\xymatrix{
A  \ar[r]^{\partial} & B_1  \ar[r]& C_1  \\
A \ar[u]^{1_A} \ar[r]^{\partial} &B \ar[u]^{f^*} \ar[r]  & C
\ar[u]^{g^*}} \quad\quad \xymatrix{
A \ar[r]^{\partial} \ar[d]_{1_A}& B_1 \ar[r] \ar[d]_{\theta}& C_1  \ar[d]_{\eta}\\
A  \ar[r]^{\partial} & B \ar[r]  & C.}
$$
Assume moreover that $\theta\circ f^*=1_{B}$,  $\eta\circ
g^*=1_{C}$, and that ${\rm{Hom}}_{\tc}(A, C_1[-1])=0$. Then we have
a \lq\lq cross\rq\rq \,isomorphism
$$
B_1\oplus C\cong B\oplus C_1.
$$
\end{lemma}

\begin{remark}\label{injectives}
If the category $\tc$ is the derived category of an abelian category
$\ac$ with enough injectives (e.g. $D^b_c(Y)$), and
$A\in{\text{Ob}}(\ac)$, and $C_1$ is a complex in degree $\geq 0$,
then the assumption ${\rm{Hom}}_{\tc}(A, C_1[-1])=0$ is verified.
\end{remark}

\begin{proof}[Proof of Lemma \ref{l3}]
Consider the following commutative diagram:
$$\xymatrix{
A \ar[r]^{1_A} \ar[d]^{\partial} &A \ar[r] \ar[d]^{\partial}&0 \ar[r] \ar[d]^{}&A\left[ 1 \right] \ar[d]\\
B \ar[r]^ {f^{*}} \ar[d] &B_1 \ar[r] \ar[d] & B_2 \ar[r] \ar[d] &B\left[ 1 \right] \ar[d]\\
C \ar[r]^{g^{*}} \ar[d] &C_1 \ar[r] \ar[d] & C_2 \ar[r] \ar[d] &C\left[ 1 \right] \ar[d]\\
A\left[ 1 \right] \ar[r]^{1_A[1]} &A\left[ 1 \right] \ar[r] &0
\ar[r] &A\left[ 2 \right]}
$$
where the first and second columns are the ones given in the
hypothesis, and the fourth column is obtained by the first one by
means of TR2. The first row, which gives the fourth one by means of
TR2, is given by TR1. The second and third rows are given by
completion of $f^{*}$ and $g^{*}$, respectively, by means of TR1.
Lastly, the arrows in the third column are given by TR3. Observe
that the third column, a priori, is not  a distinguished triangle.

Since $\theta\circ f^*=1_{B}$ and  $\eta\circ g^*=1_{C}$, by Lemma
\ref{l1} and its proof, we know that $B_1\cong B\oplus B_2$, and
that $C_1\cong C\oplus C_2$. Therefore, it suffices to prove that
$B_2\cong C_2$. To this purpose, we are going to use TR4 \cite[p.
11]{Dimca2} as follows.

Corresponding to the composition $A \stackrel
{\partial}{\rightarrow} B \rightarrow B_1$ at the top left square in
the diagram, and to the distinguished triangles given by the first
column, the second row, and the second column, TR4 says there exist
a distinguished triangle
\begin{equation}\label{secondtriangle}
C \stackrel{\gamma}\rightarrow C_1 \rightarrow B_2 \rightarrow C
\left[ 1 \right]
\end{equation}
and a triangle morphism:
$$
\xymatrix{
A \ar[r] \ar[d]^{=} &B \ar[r] \ar[d]^{f^*} & C \ar[r] \ar[d]^{\gamma} &A \left[ 1 \right] \ar[d]^{=}\\
A \ar[r] &B_1 \ar[r] &C_1 \ar[r] &A \left[ 1 \right].}
$$
The same diagram appears in our assumptions, with $g^{*}$ instead of
$\gamma$. It follows that $g^{*}=\gamma$, because
${\rm{Hom}}_{\tc}(A, C_1[-1])=0$ \cite[Proposition 1.1.9., p.
23]{BBD}. Now, comparing (\ref{secondtriangle}) with the third row
of the diagram at the beginning of the proof, we see that $B_2\cong
C_2$, because the third object in a distinguished triangle is
unique, up to isomorphism.
\end{proof}

We are in position to prove that $(i)$ implies $(ii)$ in Theorem
\ref{Crossintro}. We keep the notations introduced in the diagram
(\ref{diagram}).

First notice that the pull-back induces a natural commutative
diagram of distinguished triangles in $D^b_c(Y)$ \cite[p.
46]{Dimca2}:
\begin{equation}\label{pb}
\xymatrix{
Rf_*(\partial_{X\, !}\mathbb A_U)  \ar[r]^{\partial_X} & Rf_*\mathbb A_X  \ar[r]^{j^*}& Rf_*\mathbb A_{\widetilde W}  \\
\partial_{Y\, !}\mathbb A_U \ar[u]^{1} \ar[r]^{\partial_Y} &\mathbb
A_Y \ar[u]^{f^*} \ar[r]^{i^*}  & \mathbb A_W. \ar[u]^{g^*}}
\end{equation}
In view of Lemma \ref{l2}, the vertical map $\partial_{Y\, !}\mathbb
A_U \stackrel {1}{\longrightarrow} Rf_*(\partial_{X\, !}\mathbb
A_U)$  on the left is an isomorphism in $D^b_c(Y)$. Now consider the
following diagram:
\[
\xymatrix{ Rf_*(\partial_{X\, !}\mathbb A_U)  \ar[r]^{\partial_X}
\ar[d]_{1}&
Rf_*\mathbb A_X \ar[r]^{j^*} \ar[d]_{\theta}& Rf_*\mathbb A_{\widetilde W}  \\
\partial_{Y\, !}\mathbb A_U  \ar[r]^{\partial_Y} & \mathbb
A_Y  \ar[r]^{i^*}  & \mathbb A_W.}
\]
Since the pull-back diagram is commutative, and $\theta$ has degree
one (so $\theta\circ f^*=1_{\mathbb A_Y}$), it follows that previous
square commutes. In fact:
$$
\theta\circ\partial_X=\theta\circ\left(f^*\circ\partial_Y\circ
1\right)=\left(\theta\circ f^*\right)\circ\partial_Y\circ
1=1_{\mathbb A_Y}\circ\partial_Y\circ 1=\partial_Y\circ 1.
$$
Then, by axiom TR3,  previous diagram extends to a \lq\lq
Gysin\rq\rq \, morphism of triangles, induced  by the bivariant
class $\theta$:
\begin{equation}\label{Gy}
\xymatrix{ Rf_*(\partial_{X\, !}\mathbb A_U)  \ar[r]^{\partial_X}
\ar[d]_{1}&
Rf_*\mathbb A_X \ar[r]^{j^*} \ar[d]_{\theta}& Rf_*\mathbb A_{\widetilde W}   \ar[d]_{\eta}\\
\partial_{Y\, !}\mathbb A_U  \ar[r]^{\partial_Y} & \mathbb
A_Y  \ar[r]^{i^*}  & \mathbb A_W.}
\end{equation}
In this diagram, by \cite[loc. cit.]{BBD} (compare with Remark
\ref{injectives}), the morphism $\eta$ is unique. For the same
reason, since composing this diagram with the diagram induced by the
pull-back, we get the identity on both $\partial_{Y\, !}\mathbb A_U$
and $\mathbb A_Y $, we also have $\eta\circ g^*=1_{\mathbb A_W}$. At
this point, it is clear that the decomposition appearing in $(ii)$
follows from Lemma \ref{l3} and Remark \ref{injectives}. This
concludes the proof of Theorem \ref{Crossintro}.

\begin{remark}\label{eta}
Bivariant Theory provides a pull-back morphism $\eta_1:=i^*(\theta)$
\cite[(3), p. 19]{FultonCF}, with:
$$
\eta_1:Rf_*\mathbb A_{\widetilde W}\to \mathbb A_W.
$$
We are not able to prove that $\eta=\eta_1$, i.e. that the Gysin
diagram, with $\eta_1$ instead of $\eta$, commutes. However, we will
prove, later, that $\eta$ and $\eta_1$ induce the same morphism  in
(co)homology. Notice that also $\eta_1$ has degree one, and
therefore we also  have $\eta_1\circ g^*=1_{\mathbb A_W}$.
Therefore, if a morphism of degree one was unique, then
$\eta=\eta_1$.
\end{remark}

\bigskip
\section{Consequences for the (co)homology.}
Keep the same assumption of Theorem \ref{Crossintro}, and suppose there
is a bivariant class of degree one for $f$. Then we have a cross
isomorphism $Rf_*\mathbb A_X\oplus \mathbb A_W\cong Rf_*\mathbb
A_{\widetilde W}\oplus \mathbb A_Y$. Taking hypercohomology
(hypercohomology with compact support resp.), we deduce isomorphisms
in cohomology (Borel-Moore homology resp.):
$$
H^{\bullet}(X)\oplus H^{\bullet}(W)\cong H^{\bullet}(\widetilde
W)\oplus H^{\bullet}(Y),\quad\quad H_{\bullet}(X)\oplus
H_{\bullet}(W)\cong H_{\bullet}(\widetilde W)\oplus H_{\bullet}(Y).
$$
Using the triangle morphisms (\ref{pb}) and (\ref{Gy}), we may
explicit this isomorphisms as follows.

First, taking hypercohomology \cite[p. 46]{Dimca2}, the triangle
morphisms (\ref{pb}) and (\ref{Gy}) induce  commutative diagrams
with exact rows:
$$
\xymatrix{ H^k(X,\widetilde W) \ar[r] \ar[d]  & H^k(X) \ar[r]^{j^*}
& H^k(\widetilde W) \ar[r]^{\partial_X}
&H^{k+1}(X,\widetilde W) \ar[d]^{}\\
H^k(Y,W)  \ar[r] \ar[u]^{=}  &H^k(Y) \ar[r]^{i^*} \ar[u]^{f^*}
&H^k(W) \ar[r]^{\partial_Y} \ar[u]^{g^*} &H^{k+1}(Y, W) \ar[u]^{=}}
$$
and
$$
\xymatrix{ H^k(X,\widetilde W) \ar[r] \ar[d]^{=} & H^k(X)
\ar[r]^{j^*} \ar[d]^{\theta_*} & H^k(\widetilde W)
\ar[r]^{\partial_X} \ar[d]^{\eta_*}
&H^{k+1}(X,\widetilde W) \ar[d]^{=}\\
H^k(Y,W) \ar[r] \ar[u]^{}&H^k(Y) \ar[r]^{i^*} &H^k(W)
\ar[r]^{\partial_Y} &H^{k+1}(Y, W) \ar[u]^{}}
$$
for every $k\in\mathbb Z$. Since these diagrams commute, and
$\theta_*\circ f^*={\text{id}}_{H^{\bullet}(Y)}$ and $\eta_*\circ
g^*={\text{id}}_{H^{\bullet}(W)}$, a chase diagram shows that the
sequence:
$$
0\to H^k(X)\stackrel{\alpha^*}\to H^k(\widetilde W)\oplus
H^k(Y)\stackrel{\beta^*}\to H^k(W)\to 0,
$$
with
$$
\alpha^*(x):=(j^*(x),\, -\theta_*(x)), \quad \beta^*(\widetilde
w,\,y):=\eta_*(\widetilde w)+i^*(y),
$$
is exact (compare with \cite[Proposition 6.7, (e), p.
114-115]{FultonIT}). Moreover, the map
$$
w\in H^k(W) \to (g^*(w),0)\in H^k(\widetilde W)\oplus H^k(Y)
$$
is a right section for the sequence, and so we get an explicit
isomorphism:
\begin{proposition}\label{crossexpl}
The map
$$
\varphi^*:H^k(X)\oplus H^k(W)\to H^k(\widetilde W)\oplus H^k(Y),
$$
with
$$
\varphi^*(x,\, w):=(j^*(x)+g^*(w),\,-\theta_*(x)),
$$
is an isomorphism.
\end{proposition}
We may interpret the map $\varphi^*$ as a matrix product (compare
with \cite[p. 328]{Jouanolou}):
$$
\left[\begin{matrix} \widetilde w \\
y
\end{matrix}\right]=\left[\begin{matrix} j^* & g^*\\
-\theta_* & 0
\end{matrix}\right]\cdot \left[\begin{matrix} x \\
w
\end{matrix}\right].
$$
Since
$$
\varphi^*(-f^*y,\,i^*y)=(0,y),
$$
the matrix defining the inverse map $({\varphi^*})^{-1}$ has the
following form:
$$
\left[\begin{matrix} x \\
w
\end{matrix}\right]=\left[\begin{matrix} \lambda_* & -f^*\\
\mu_* & i^*
\end{matrix}\right]\cdot \left[\begin{matrix} \widetilde w \\
y
\end{matrix}\right],
$$
where the functions:
$$
\lambda_*: H^{\bullet}(\widetilde W)\to H^{\bullet}(X),\quad \mu_*:
H^{\bullet}(\widetilde W)\to H^{\bullet}(W)
$$
are uniquely determined by the condition that the two matrices above
are the inverse each other, i.e. by the equations:
\begin{equation}\label{eqc1}
\begin{cases}
\lambda_*\circ
j^*+f^*\circ\theta_*={\rm{id}_{H^{\bullet}(X)}}\\
\lambda_*\circ g^*=0\\
\mu_*\circ j^*-i^*\circ \theta_*=0\\
\mu_*\circ g^*={\rm{id}_{H^{\bullet}(W)}},
\end{cases}
\end{equation}
which in turn are equivalent to the equations:
\begin{equation}\label{eqc2}
\begin{cases}
j^*\circ\lambda_*+g^*\circ\mu_*={\rm{id}_{H^{\bullet}(\widetilde
W)}}\\
\theta_*\circ \lambda_*=0\\
j^*\circ f^*=g^*\circ i^*\\
\theta_*\circ f^*={\rm{id}_{H^{\bullet}(Y)}}.
\end{cases}
\end{equation}
Since we also have $\eta_*\circ j^*-i^*\circ \theta_*=0$ and
$\eta_*\circ g^*={\rm{id}_{H^{\bullet}(W)}}$, by the uniqueness, it
follows that $\eta_*=\mu_*$.

\smallskip
\begin{remark}\label{eta2}
Let $\eta_1:=i^*(\theta)$ be the pull-back of $\theta$ on $W$. By
properties of bivariant classes \cite[(G2), p. 26]{FultonCF}, we see
that $(\eta_1)_*\circ j^*-i^*\circ \theta_*=0$ and $(\eta_1)_*\circ
g^*={\rm{id}_{H^{\bullet}(W)}}$. As before, this proves that
$\eta_*=(\eta_1)_*$. Similarly, for the maps induced in homology,
one sees that $\eta^*=(\eta_1)^*$ (see below).
Recall that we do not know whether
$\eta=\eta_1$ (compare with Remark \ref{eta}).
\end{remark}

Using these equations, we are able to explicit also the isomorphism
induced in cohomology by the decomposition appearing in $(iii)$ of
Theorem \ref{Crossintro}. First observe that, since $\eta_*\circ
g^*={\rm{id}_{H^{\bullet}(W)}}$, we may see $H^k(W)$, via $g^*$,  as
a direct summand of $H^k(\widetilde W)$ for every integer $k$. Denote by
$$\frac{H^k(\widetilde W)}{H^k(W)}$$ the corresponding quotient.
\begin{proposition}\label{decexpl}
For every $k$, the map
$$
x\in H^k(X)\to (\theta_*x,\, j^*x)\in H^k(Y)\oplus
\left[\frac{H^k(\widetilde W)}{H^k(W)}\right]
$$
is an isomorphism, whose inverse is the map
$$
(y,\, \widetilde w)\in H^k(Y)\oplus \left[\frac{H^k(\widetilde
W)}{H^k(W)}\right]\to f^*(y)+\lambda_* \widetilde w\in H^k(X).
$$
\end{proposition}
\begin{proof} First observe that the map
$$
x\in H^k(X)\to (\theta_*x,\, x-f^*\theta_*x)\in H^k(Y)\oplus \ker
\theta_*
$$
is an isomorphism. Next, observe that previous equations (\ref{eqc1}) and (\ref{eqc2})
imply that
$j^*$ induces an isomorphism
$$
j^*: \ker \theta_* \to \ker \eta_*,
$$
whose inverse acts as $\lambda_*$. On the other hand, we also have
an isomorphism:
$$
\widetilde w\in  \ker\eta_*\to\widetilde w\in \frac{H^k(\widetilde
W)}{H^k(W)}.
$$
\end{proof}

Similarly, taking hypercohomology with compact support, the triangle
morphisms (\ref{pb}) and (\ref{Gy}) induce commutative diagrams with
exact rows involving Borel-Moore homology:
$$
\xymatrix{ H_{k+1}(U) \ar[r]^{\partial_X} \ar[d]^{=} &
H_k(\widetilde W) \ar[r]^{j_*} \ar[d]^{g_*} & H_k(X) \ar[r]
\ar[d]^{f_*}
&H_{k}(U) \ar[d]^{=}\\
H_{k+1}(U) \ar[r]^{\partial_Y} \ar[u]^{}&H_k(W) \ar[r]^{i_*} &H_k(Y)
\ar[r] &H_{k}(U) \ar[u]^{}}
$$
and
$$
\xymatrix{ H_{k+1}(U) \ar[r]^{\partial_X} \ar[d]  & H_k(\widetilde
W) \ar[r]^{j_*} & H_k(X) \ar[r]
&H_{k}(U) \ar[d]^{}\\
H_{k+1}(U)  \ar[r]^{\partial_Y} \ar[u]^{=}  &H_k(W) \ar[r]^{i_*}
\ar[u]^{\eta^*} &H_k(Y) \ar[r] \ar[u]^{\theta^*} &H_k(U) \ar[u]^{=}}
$$
for every $k\in\mathbb Z$. Since these diagrams commute, and
$f_*\circ \theta^*={\text{id}}_{H_{\bullet}(Y)}$ and $g_*\circ
\eta^*={\text{id}}_{H_{\bullet}(W)}$, a chase diagram shows that the
sequence:
$$
0\to H_k(W)\stackrel{\alpha_*}\to H_k(\widetilde W)\oplus
H_k(Y)\stackrel{\beta_*}\to H_k(X)\to 0,
$$
with
$$
\alpha_*(w):=(\eta^*(w),\, -i_*(w)), \quad \beta_*(\widetilde
w,\,y):=j_*(\widetilde w)+\theta^*(y),
$$
is exact (compare with \cite[pp. 264-266, Proposition 2.5]{RCMP}).
Moreover, the map
$$
(\widetilde w,\, y)\in H_k(\widetilde W)\oplus H_k(Y)\to
g_*\widetilde w\in H_k(W)
$$
is a left section for the sequence, and so we get an explicit
isomorphism:
\begin{proposition}\label{crossexpl2}
The map
$$
\varphi_*:H_k(\widetilde W)\oplus H_k(Y)\to H_k(X)\oplus H_k(W),
$$
with
$$
\varphi_*(\widetilde w,\, y):=(j_*(\widetilde
w)+\theta^*(y),\,g_*(\widetilde w)),
$$
is an isomorphism.
\end{proposition}
We may interpret the map $\varphi_*$ as a matrix product:
$$
\left[\begin{matrix} x \\
w
\end{matrix}\right]=\left[\begin{matrix} j_* & \theta^*\\
g_* & 0
\end{matrix}\right]\cdot \left[\begin{matrix} \widetilde w \\
y
\end{matrix}\right].
$$
Since
$$
\varphi_*(\eta^*w,\,-i_*w)=(0,w),
$$
the matrix defining the inverse map $(\varphi_*){^{-1}}$ has the
following form:
$$
\left[\begin{matrix} \widetilde w \\
y
\end{matrix}\right]=\left[\begin{matrix} \lambda^* & \eta^*\\
\mu^* & -i_*
\end{matrix}\right]\cdot \left[\begin{matrix} x \\
w
\end{matrix}\right],
$$
where the functions:
$$
\lambda^*: H_{\bullet}(X)\to H_{\bullet}(\widetilde W), \quad \mu^*:
H_{\bullet}(X)\to H_{\bullet}(Y)
$$
are uniquely determined by the condition that the two matrices above
are the inverse each other, i.e. by the equations:
\begin{equation}\label{eqo1}
\begin{cases}
j_*\circ \lambda^*+\theta^*\circ \mu^*={\rm{id}_{H_{\bullet}(X)}}\\
j_*\circ \eta^* -\theta^*\circ i_*=0\\
g_*\circ \lambda^*=0\\
g_*\circ \eta^*={\rm{id}_{H_{\bullet}(W)}},
\end{cases}
\end{equation}
which in turn are equivalent to the equations:
\begin{equation}\label{eqo2}
\begin{cases}
\lambda^*\circ j_*+\eta^*\circ g_*={\rm{id}_{H_{\bullet}(\widetilde
W)}}\\
\lambda^*\circ \theta^*=0\\
\mu^*\circ j_*=i_*\circ g_*\\
\mu^*\circ \theta^*={\rm{id}_{H^{\bullet}(Y)}}.
\end{cases}
\end{equation}
In particular, it follows that $\mu^*=f_*$. Using these equations,
we are able to explicit the isomorphism induced in Borel-Moore
homology by  $(iii)$ of Theorem \ref{Crossintro}. First, observe
that, since $g_*\circ\,\eta^*={\rm{id}_{H_{\bullet}(W)}}$, we may
see $H_k(W)$, via $\eta^*$,  as a direct summand of $H_k(\widetilde
W)$ for every integer $k$. Denote by
$$\frac{H_k(\widetilde W)}{H_k(W)}$$ the corresponding quotient.
\begin{proposition}\label{decexpl2}
For every $k$, the map
$$
x\in H_k(X)\to (f_*x,\, \lambda^* x)\in H_k(Y)\oplus
\left[\frac{H_k(\widetilde W)}{H_k(W)}\right]
$$
is an isomorphism, whose inverse is the map
$$
(y,\, \widetilde w)\in H_k(Y)\oplus \left[\frac{H_k(\widetilde
W)}{H_k(W)}\right]\to \theta^*(y)+j_*\lambda^* j_* \widetilde w\in
H_k(X).
$$
\end{proposition}
\begin{proof} First observe that the map
$$
x\in H_k(X)\to (f_*x,\, x-\theta^*f_*x)\in H_k(Y)\oplus \ker f_*
$$
is an isomorphism. Next, observe that previous equations (\ref{eqo1}) and (\ref{eqo2}) imply that
$\lambda^*$ induces an isomorphism
$$
\lambda^*: \ker f_* \to \ker g_*,
$$
whose inverse acts as $j_*$. On the other hand, we also have an
isomorphism:
$$
\widetilde w\in  \ker g_*\to\widetilde w\in \frac{H_k(\widetilde
W)}{H_k(W)}.
$$
\end{proof}

\bigskip
\section{Behaviour under the duality morphism.}
One may ask how previous decompositions given in Proposition
\ref{decexpl} and Proposition \ref{decexpl2}, behave under the cap
product with a homology class. In this section  we consider only the
case of the fundamental class, and algebraic maps.

\smallskip
Consider a map $f:X\to Y$ as in Theorem \ref{Crossintro}, and assume
there exists a bivariant class of $f$ of degree one. Moreover,
assume that $f$ is onto, and that $X$ and $Y$ are open subsets of complex quasi-projective
varieties of the same complex dimension $n$. Let $[X]\in H_{2n}(X)$
be the fundamental class of $X$, and consider the map
\begin{equation}\label{dm}
\mathcal D_X: x\in H^{k}(X)\to x\cap [X]\in H_{2n-k}(X)
\end{equation}
given by the cap product with $[X]$. When $X$ is a circuit, this map
is called {\it the duality morphism} \cite[p. 150]{McCrory}. If, in
addition, $X$ is smooth, then $\mathcal D_X$ is the Poincar\'e
Duality isomorphism. In view of the decompositions given in
Proposition \ref{decexpl} and Proposition \ref{decexpl2}, the map
$\mathcal D_X$ identifies with a map
$$
\mathcal D_X:H^k(Y)\oplus \left[\frac{H^k(\widetilde
W)}{H^k(W)}\right]\to H_{2n-k}(Y)\oplus
\left[\frac{H_{2n-k}(\widetilde W)}{H_{2n-k}(W)}\right]
$$
which acts as follows:
$$
\mathcal D_X(y,\,\widetilde w)=(f_*([X]\cap(f^*y+\lambda_*\widetilde
w)),\, \lambda^*([X]\cap(f^*y+\lambda_*\widetilde w))).
$$
The map $\mathcal D_X$ induces two projections
$$
P_1:y\in H^{k}(Y)\to f_*([X]\cap f^*y)\in H_{2n-k}(Y),
$$
$$
P_2:\widetilde w\in \left[\frac{H^k(\widetilde W)}{H^k(W)}\right]\to
\lambda^*([X]\cap \lambda_*\widetilde w)\in
\left[\frac{H_{2n-k}(\widetilde W)}{H_{2n-k}(W)}\right].
$$
Observe that, by the projection formula \cite[p. 24]{FultonCF}, we have
$$
 f_*([X]\cap f^*y)=[Y]\cap y.
$$
Therefore, $P_1= \mathcal D_Y$, i.e. $P_1$ is nothing but the
duality morphism on $Y$.

\begin{corollary}\label{duality}
The duality morphism $\mathcal D_X:H^{k}(X)\to H_{2n-k}(X)$ is the
direct sum of $\mathcal D_Y$ and $P_2$, i.e.
$$
\mathcal D_X=\mathcal D_Y\oplus P_2.
$$
\end{corollary}

\begin{proof}
We have to prove that:

\medskip
$\bullet$) for every $\widetilde w\in \frac{H^i(\widetilde
W)}{H^i(W)}$ one has $f_*([X]\cap \lambda_* \widetilde w)=0$, and

\smallskip
$\bullet$) for every $y\in H^i(Y)$ one has $\lambda^*([X]\cap
f^*y)=0$.

\medskip
To this purpose, first observe that $\theta^*([Y])=[X]$, i.e. the
Gysin map sends the fundamental class of $Y$ in the fundamental
class of $X$. In fact, from the equations (\ref{eqo1}) we obtained in homology
(recall that $\mu^*=f_*$),
we know that
$\theta^*([Y])=\theta^*f_*[X]=[X]-(j_*\circ\lambda^*)([X])=[X]$
because $\lambda^*[X]=0\in H_{2n}(\widetilde W)=\{0\}$ for
dimensional reasons.

\medskip
 $\bullet$) Now, by \cite[p. 26, G4, (ii)]{FultonCF}, we have:
$$
f_*([X]\cap \lambda_* \widetilde w)=f_*(\theta^*[Y]\cap \lambda_*
\widetilde w)=(\theta_*\lambda_*\widetilde w)\cap [Y]
$$
which is zero because, from the equations (\ref{eqc2}) we obtained in cohomology,
we know that $\theta_*\circ\lambda_*=0$ .

\smallskip
$\bullet$) Next, by \cite[p. 26, G4, (iii)]{FultonCF}, we have:
$$\lambda^*([X]\cap f^*y)=
\lambda^*(\theta^*[Y]\cap f^*y)=\lambda^*(\theta^*(Y\cap y))
$$
which is zero because, from the equations (\ref{eqo2}) we obtained in homology,
we know that  $\lambda^*\circ \theta^*=0$.
\end{proof}

\bigskip
\section{Resolution of singularities of a homology manifold.}

In this section we are going to prove Theorem \ref{homology} stated
in the Introduction. Observe that it applies to a resolution of
singularities of $Y$.

\medskip
First assume that $Y$ is an $\mathbb
A$-homology manifold.

By \cite[Definition 3.1, Theorem 3.7]{Brasselet}, we know that the
fundamental class of $Y$
$$
[Y]\in H_{2n}(Y)\cong H^{-2n}(Y\stackrel{}\to pt)
$$
is a strong orientation. Therefore, we have
$$
Hom_{D^{b}_{c}(Y)}(Rf_*\mathbb A_X, \mathbb A_Y)\cong
H^0(X\stackrel{f}\to Y)\stackrel{\bullet [Y]}\cong
H^{-2n}(X\stackrel{}\to pt)\cong H_{2n}(X)\cong H^0(X).
$$
Since $f$ is birational, the bivariant class corresponding to
$1_X\in H^0(X)$ is a bivariant class of degree one for $f$, and it
is unique (compare with Remark \ref{remintro}, $(ii)$ and $(v)$). By
Theorem \ref{Crossintro}, we know there exists a decomposition
\begin{equation}\label{D1}
Rf_*\mathbb A_X[n] \cong \mathbb A_Y[n] \oplus \mathcal K[n].
\end{equation}
It is clear that $\mathcal K$ is supported on $W$. Passing to
Verdier dual, we get:
\begin{equation}\label{D2}
D\left(Rf_*\mathbb A_X[n]\right) \cong D\left(\mathbb A_Y[n]\right)
\oplus D\left(\mathcal K[n]\right). \end{equation} Now let $$[X]\in
H_{2n}(X)$$ be the fundamental class of $X$. We have \cite[p.
804-805]{Brasselet}:
$$
[X]\in H_{2n}(X)\cong H^{-2n}(X\stackrel{}\to pt.)\cong
Hom_{D^{b}_{c}(X)}(\mathbb A_X[n],D\left(\mathbb A_X[n]\right)).
$$
Therefore, $[X]$ corresponds to a morphism
\begin{equation}\label{is}
\mathbb A_X[n]\to D\left(\mathbb A_X[n]\right),
\end{equation}
whose induced map in hypercohomology is nothing but the duality
morphism (\ref{dm}). If we assume that $X$ is an $\mathbb
A$-homology manifold, the morphism (\ref{is}) is an isomorphism
\cite[Proof of Theorem 3.7]{Brasselet}. Since $D\left(Rf_*\mathbb
A_X[n]\right)\cong Rf_* D\left(A_X[n]\right)$ \cite[p. 69]{Dimca2},
it induces an isomorphism
\begin{equation*}
Rf_*\mathbb A_X[n]\to D\left(Rf_*\mathbb A_X[n]\right),
\end{equation*}
which in turn, via the previous decompositions (\ref{D1}) and
(\ref{D2}), induces two projections
$$
\mathbb A_Y[n] \to D\left(\mathbb A_Y[n]\right), \quad  \mathcal
K[n]\to D\left( \mathcal K[n]\right).
$$
By Corollary \ref{duality}, we know that the maps induced in
hypercohomology by $\mathcal K[n]\to D\left( \mathcal K[n]\right)$
are isomorphisms, and this holds true when restricting to every open
subset of $Y$. Therefore, we have $ \mathcal K[n]\cong
D\left(\mathcal K[n]\right)$, i.e. $\mathcal K[n]$ is self-dual.

\medskip
Conversely, assume  there exists a bivariant class $\theta$ of
degree one. Arguing as before, by Corollary \ref{duality}, we know
that the isomorphism (\ref{is}) induces an isomorphism $\mathbb
A_Y[n]\cong D\left(\mathbb A_Y[n]\right)$. This is equivalent to say
that $Y$ is an $\mathbb A$-homology manifold \cite[loc.
cit.]{Brasselet}.

This concludes the proof of Theorem \ref{homology}.

\begin{remark}\label{remfin}
$(i)$ With the notations as in Theorem \ref{homology}, assume  there
exists a bivariant class $\theta$ of degree one. When the
coefficients are in a field, we may prove that $Y$ is an $\mathbb
A$-homology manifold in a different manner, using the Decomposition
Theorem \cite[p. 161]{Dimca2}. In fact, by the Decomposition
Theorem, there exists a certain decomposition
$$Rf_*\mathbb A_X[n] \cong  IC_Y^{\bullet} \oplus \mathcal H.$$
Comparing with the decomposition given by Theorem \ref{Crossintro}
$$Rf_*\mathbb A_X[n] \cong \mathbb A_Y[n] \oplus \mathcal K[n],$$
it follows a non-zero endomorphism $IC_Y^{\bullet}\to \mathbb
A_Y[n]\to IC_Y^{\bullet}$. On the other hand,  $IC_Y^{\bullet}$
belongs to the core of $D^b_c(Y)$, which is an abelian subcategory
of $D^b_c(Y)$. In this category, $IC_Y^{\bullet}$ is a simple
object. Therefore, by Schur's Lemma, the composition
$IC_Y^{\bullet}\to \mathbb A_Y[n]\to IC_Y^{\bullet}$ is an
automorphism. Observe that also the composition $\mathbb A_Y[n]\to
IC_Y^{\bullet}\to \mathbb A_Y[n]$ is an automorphism, because
$Hom_{D^{b}_{c}(Y)}(\mathbb A_Y, \mathbb A_Y)\cong H^0(Y)$. So,
$IC_Y^{\bullet}\cong \mathbb A_Y[n]$.

\smallskip
$(ii)$ Since $\mathcal K[n]$ is self-dual, it follows that
$$
h^{2n-i}(\widetilde W)-h_i(\widetilde W)=h^{2n-i}(W)-h_i(W)
$$
for every $i\in\mathbb Z$.

\smallskip
$(iii)$ The following example shows there exist projective
birational maps $f:X\to Y$ such that $H^0(X\stackrel{f}\to Y)\neq
0$, without bivariant classes of degree one. The coefficients are in
$\mathbb Q$.

Let $C\subset \mathbb P^3$ be a projective non-singular curve of
genus $g\geq 1$. Let $Y\subset \mathbb P^4$ be the cone over $C$,
and let $f:X\to Y$ be the blowing-up of $Y$ at the vertex $y\in Y$.
By the Decomposition Theorem (see e.g. \cite{DGF3}) we have
$$
Rf_*\mathbb Q_X=\mathbb Q_y[-2]\oplus IC_Y^{\bullet}[-2].
$$
On the other hand, combining \cite[9.13, p. 128]{Iversen} with
\cite[Remark 2.4.5, (i), p. 46]{Dimca2}, we have
$$
Hom_{D^b_c(Y)}(\mathbb Q_y, \mathbb Q_Y[2])\cong
H^2(Y,Y\backslash\{y\})\cong H^1(L),
$$
where $L$ is the link of $Y$ at the vertex $y$. The Hopf fibration
$L\to C$ induces a Gysin sequence
$$
0\to H^1(C)\to H^1(L)\to H^0(C)\to H^2(C)\to\dots
$$
from which we get $h^1(L)=h^1(C)=2g\geq 2$. It follows that
$H^0(X\stackrel{f}\to Y)\cong Hom_{D^{b}_{c}(Y)}(Rf_*\mathbb A_X,
\mathbb A_Y)\neq 0$, and that $Y$ is not a homology manifold. In
particular, since $X$ is smooth, in view of Theorem \ref{homology},
there is no a bivariant class of degree one.
\end{remark}

\begin{corollary}\label{strong}
Let $f:X\to Y$ be a projective birational morphism between
irreducible and quasi-projective complex varieties of the same
complex dimension $n$. Let $\theta\in H^0(X\stackrel{f}\to Y)$ be a
bivariant class. If $\theta$ is a strong orientation for $f$, then
$\theta$ is a bivariant class of degree one for $f$, up to
multiplication by a unit. Moreover, if $X$ is an $\mathbb
A$-manifold and $\theta$ is a bivariant class of degree one for $f$,
then $\theta$ is a strong orientation for $f$.
\end{corollary}

\begin{proof} First assume that $\theta$ is a strong orientation for
$f$.

Let $U\subset Y$ be a Zariski non-empty open subset of $Y$ such that
$f^{-1}(U)\cong U$ via $f$.  Product by $\theta$ gives an
isomorphism:
$$
H^0(f^{-1}(U)\to X)\stackrel{\bullet \theta}\to H^0(U\to Y).
$$
On the other hand, by Verdier Duality \cite[p. 803]{Brasselet}, and
\cite[Corollary 3.2.12., p. 65]{Dimca2}, we have:
$$
H^0(f^{-1}(U)\to X)\cong H^0(f^{-1}(U)), \quad {\text{and}}\quad
H^0(U\to Y)\cong H^0(U).
$$
Therefore, $\theta$ induces an isomorphism $H^0(f^{-1}(U))\to
H^0(U)$. It follows that, up to multiplication by a unit, $\theta$
is a bivariant class of degree one.

Conversely, assume $X$ is an $\mathbb A$-manifold, and $\theta$ is a
bivariant class of degree one for $f$.

In this case, by Theorem \ref{homology}, we know that also $Y$ is an
$\mathbb A$-homology manifold, and that $\theta$ corresponds to
$1_X$ in the isomorphism $H^0(X\stackrel{f}\to Y)\cong H^0(X)$.
Since $X$ and $Y$ are $\mathbb A$-manifolds,  we get:
$$
f^{!}(\mathbb A_Y)=D(f^*(D(\mathbb A_Y)))=D(f^*(\mathbb
A_Y[2n]))=D(\mathbb A_X[2n])=\mathbb A_X.
$$
Therefore, $\theta$ corresponds to an isomorphism in
$$
Hom_{D^b_c(X)}(\mathbb A_X,f^{!}\mathbb A_Y)\cong
Hom_{D^b_c(X)}(\mathbb A_X,\mathbb A_X)\cong H^0(X).
$$
By \cite[7.3.2, proof of Proposition, p. 85]{FultonCF}, we deduce
that $\theta$ is a strong orientation for $f$.
\end{proof}

\begin{proposition}\label{strong2}
Let $f:X\to Y$ be a projective birational morphism between
irreducible and quasi-projective complex varieties of the same
complex dimension $n$. Let $\theta\in H^0(X\stackrel{f}\to Y)$ be a
bivariant class. If $\theta$ is a strong orientation for $f$, and
$Y$ is an $\mathbb A$-homology manifold, then also $X$ is so.
\end{proposition}

\begin{proof} Since $Y$ is an $\mathbb A$-homology manifold, we have:
$$
f^{!}(\mathbb A_Y)=D(f^*(D(\mathbb A_Y)))=D(f^*(\mathbb
A_Y[2n]))=D(\mathbb A_X[2n]).
$$
On the other hand, if $\theta$ is a strong orientation, then
\cite[loc. cit.]{FultonCF}
$$
f^{!}(\mathbb A_Y)\cong \mathbb A_X.
$$
Therefore, we get $D(\mathbb A_X[2n])\cong \mathbb A_X$. This means
that $\mathbb A_X[n]$ is self-dual, i.e. $X$ is an $\mathbb
A$-homology manifold \cite[proof of Theorem 3.7]{Brasselet}.
\end{proof}

\begin{remark}\label{strong3}
Let $f:X\to Y$ be a birational, projective local complete
intersection morphism between complex irreducible quasi-projective
algebraic varieties. Let $\theta\in H^0(X\stackrel{f}\to Y)$ be the
{orientation class} of $f$. Then $\theta$ has degree one (Remark
\ref{remintro}, $(iii)$). But, in general, in view of previous
Proposition \ref{strong2}, $\theta$ cannot be a strong orientation.
\end{remark}

\end{document}